\documentclass[12pt]{amsart}

\usepackage{amsmath,amsthm,amssymb}

\usepackage{a4wide}
\usepackage{enumerate}

\newtheorem{theorem}{Theorem}[section]
\newtheorem{lemma}[theorem]{Lemma}
\newtheorem{proposition}[theorem]{Proposition}
\newtheorem{corollary}[theorem]{Corollary}
\newtheorem*{theoremA}{Theorem A}
\newtheorem*{theoremB}{Theorem B}

\theoremstyle{remark}
\newtheorem{remark}[theorem]{Remark}

\newcommand{\C}{\ensuremath{\mathbb{C}}}
\newcommand{\R}{\ensuremath{\mathbb{R}}}
\newcommand{\Pa}{\ensuremath{\mathcal{P}}}
\newcommand{\g}[1]{\ensuremath{\mathfrak{#1}}}
\newcommand{\II}{\ensuremath{I\! I}}
\DeclareMathOperator{\tr}{tr}

\begin{document}
\title[Isoparametric submanifolds]{Isoparametric submanifolds in\\ two-dimensional complex space forms}

\author[J.\ C.\ D\'{\i}az-Ramos]{Jos\'{e} Carlos D\'{\i}az-Ramos}
\author[M.\ Dom\'{\i}nguez-V\'{a}zquez]{Miguel Dom\'{\i}nguez-V\'{a}zquez}
\author[C.\ Vidal-Casti\~{n}eira]{Cristina Vidal-Casti\~{n}eira}

\address{Department of Geometry and Topology, Universidade de Santiago de Compostela, Spain}
\email{josecarlos.diaz@usc.es}

\address{ICMAT - Instituto de Ciencias Matem\'aticas (CSIC-UAM-UC3M-UCM), Madrid, Spain.}
\email{miguel.dominguez@icmat.es}

\address{Department of Geometry and Topology, Universidade de Santiago de Compostela, Spain}
\email{cristina.vidal@usc.es}

\thanks{The second author has been supported by a Juan de la Cierva-formaci\'on fellowship (Spain) and by the
ICMAT Severo Ochoa project SEV-2015-0554 (MINECO, Spain). All authors have been supported by
projects EM2014/009, GRC2013-045 and MTM2013-41335-P with FEDER funds (Spain).}

\begin{abstract}
We show that an isoparametric submanifold of a complex hyperbolic plane, according to the definition of Heintze, Liu and Olmos', is an open part of a principal orbit of a polar action.

We also show that there exists a non-isoparametric submanifold of the complex hyperbolic plane that is isoparametric according to the definition of Terng's. Finally, we classify Terng-isoparametric submanifolds of two-dimensional complex space forms.
\end{abstract}




\maketitle

\section{Introduction}

A submanifold $M$ of a Riemannian manifold $\bar{M}$ is said to be isoparametric according to Heintze, Liu and Olmos~\cite{HLO06}, henceforth simply \emph{isoparametric}, if its normal bundle $\nu M$ is flat, all nearby parallel submanifolds have constant mean curvature in the radial directions, and for any $p\in M$ there exists a totally geodesic submanifold $\Sigma_p$ through $p$ such that $T_p\Sigma_p=\nu_p M$.

We denote by $\bar{M}^2(c)$ a 2-dimensional complex space form of constant holomorphic sectional curvature $c\neq 0$. Thus, $\bar{M}^2(c)$ is a complex projective plane $\C P^2$ if $c>0$, or a complex hyperbolic plane $\C H^2$ if $c<0$. The first main result of this paper is:

\begin{theoremA}
An isoparametric submanifold of $\bar{M}^2(c)$ is congruent to an open part of a principal orbit of a polar action on $\bar{M}^2(c)$.
\end{theoremA}

The classification of isoparametric submanifolds for complex projective spaces $\C P^n$, $n\neq 15$, has been obtained in much greater generality using a different method in~\cite{Do16}. Here we deal with both cases simultaneously and obtain the result for $\C H^2$. We prove this theorem in Section~\ref{sec:theoremA}.

Recall that an isometric action of a Lie group on a Riemannian manifold is called \emph{polar} if there exists a submanifold $\Sigma$ (called section) that intersects all the orbits of the action, and such that $\Sigma$ is orthogonal to the orbits at intersection points. Polar actions on complex projective spaces have been classified in~\cite{PT99}, and polar actions on the complex hyperbolic plane have been classified in~\cite{BD11a}. See also~\cite{DDK} for the more general classification in~$\C H^n$. Therefore, our result implies the classification of isoparametric submanifolds in $\bar{M}^2(c)$ of arbitrary codimension. The classification for $\C H^2$ seems to be the first one of these characteristics in a symmetric space of noncompact type and nonconstant curvature.

A submanifold $M$ of a Riemannian manifold $\bar{M}$ is called \emph{Terng-isoparametric} if it has flat normal bundle and the eigenvalues of the shape operator with respect to any parallel normal vector field are constant. In our setting, Terng's definition is less rigid than Heintze, Liu and Olmos', and thus, a new example appears in codimension two:

\begin{theoremB}
A submanifold of $\bar{M}^2(c)$ is Terng-isoparametric if and only if it is congruent to an open part of:
\begin{enumerate}[{\rm (i)}]
\item an isoparametric submanifold of $\bar{M}^2(c)$, or
\item a Chen's surface in $\C H^2$, or
\item a circle in $\bar{M}^2(c)$.
\end{enumerate}
\end{theoremB}

The proof of Theorem~B is given in Section~\ref{sec:theoremB}. Apart from circles, which are trivial examples of Terng-isoparametric submanifolds, we do not get new examples in complex projective planes. However, there exists a Terng-isoparametric submanifold in $\C H^2$ that is neither a circle nor a principal orbit of a polar action. We have called this new example  \emph{Chen's surface}, which is homogeneous, that is, an orbit of an isometric action on the ambient space, and unique up to isometric congruence (see~\S\ref{sec:Chen}). It was introduced by Chen in~\cite{C98}, and a geometric characterization was given in~\cite{CT00}. In Section~\ref{sec:Chen} we present a new Lie  theoretic description of this submanifold in terms of the root space decomposition of the Lie algebra of the isometry group of $\C H^2$.\medskip

The motivation for this paper comes from the study of isoparametric submanifolds in symmetric spaces. The history of isoparametric submanifolds can be traced back at least to the works of Somigliana~\cite{So19} and Segre~\cite{Se38} who classified isoparametric hypersurfaces in Euclidean spaces.
Thorbergsson showed in~\cite{Th91} that compact, full and irreducible isoparametric submanifolds of codimension greater than~$2$ in Euclidean spaces are homogeneous, which implies that such submanifolds are principal orbits of polar actions, which in turn correspond to isotropy representations of symmetric spaces~\cite{Da85}.

Thorbergsson's remarkable result~\cite{Th91} readily implies the classification of isoparametric submanifolds of codimension $\geq 2$ in spheres. However, the classification of isoparametric hypersurfaces in spheres is open and still an active topic of research. See~\cite{Th10} for a recent survey on this and other related topics.

Isoparametric hypersurfaces in real hyperbolic spaces were classified by Cartan~\cite{Ca38}, whereas for higher codimension, Wu~\cite{Wu92} reduced the classification problem to that of isoparametric hypersurfaces in spheres. We highlight that, in real space forms, homogeneous isoparametric submanifolds are always principal orbits of polar actions.

The general study of isoparametric submanifolds was started by Terng~\cite{T85}, whose definition was given for spaces of constant curvature. Nowadays the general definition of isoparametric submanifold is credited to Heintze, Liu and Olmos~\cite{HLO06}. This is the notion that we use in this paper, although we also consider Terng's definition, which turns out to be less rigid when the ambient space is a complex hyperbolic plane.
This contrasts with the situation in real space forms, where both definitions agree.

Isoparametric submanifolds of complex projective spaces $\C P^n$ have been studied by the second author, who gave a classification if $n\neq 15$.
In this paper we also study Terng-isoparametric submanifolds of $\C P^2$ and conclude that no new interesting examples arise.

The classification of isoparametric hypersurfaces in complex hyperbolic spaces has recently been obtained in~\cite{DDS}.
For higher codimension the problem seems to be much more complicated. We restrict to $\C H^2$ in this paper and show that all examples are open parts of principal orbits of polar actions on $\C H^2$. Surprisingly, unlike in real space forms, there is a Terng-isoparametric submanifold of codimension~$2$ that is not isoparametric; this submanifold is homogeneous but not an orbit of a polar action.


\section{Preliminaries}\label{sec:preliminaries}

We start with some basic definitions and notations.

\subsection{Submanifold geometry}

We denote by $\bar{M}^2(c)$ a complex space form of dimension $2$ and constant holomorphic curvature $c\neq 0$. Thus, $\bar{M}^2(c)$ is isometric to a complex projective plane $\C P^2$ endowed with the Fubini-Study metric of constant holomorphic sectional curvature $c>0$, or to a complex hyperbolic plane $\C H^2$ endowed with the Bergman metric of constant holomorphic sectional curvature $c<0$. We denote by $\langle\,\cdot\,,\,\cdot\,\rangle$ the Riemannian metric of $\bar{M}^2(c)$ and by $\bar{\nabla}$, $\bar{R}$ and $J$ its Levi-Civita connection, its curvature tensor and its complex structure, respectively. Thus,
\[
\bar{R}(X,Y)Z=\frac{c}{4}(\langle Y,Z\rangle X-\langle X,Z\rangle Y
+\langle JY,Z\rangle JX-\langle JX,Z\rangle JY-2\langle JX,Y\rangle JZ),
\]
for vector fields $X$, $Y$, $Z\in\Gamma(\bar{M}^2(c))$.

Now let $M$ be a submanifold of $\bar{M}^2(c)$. We denote its normal bundle by $\nu M$, and by $\nabla$ and $R$ its Levi-Civita connection and its curvature tensor, respectively. The extrinsic geometry of $M$ is determined by its second fundamental form $\II$, which is defined by the formula $\bar{\nabla}_X Y=\nabla_X Y+\II(X,Y)$, for $X$, $Y\in\Gamma(TM)$. If $\xi\in\Gamma(\nu M)$ is a normal vector, then the shape operator $S_\xi$ with respect to $\xi$ is defined by $\langle S_\xi X,Y\rangle=\langle\II(X,Y),\xi\rangle$. We also denote by $\nabla^\perp$ the normal connection of the normal bundle $\nu M$, which is related to the shape operator via the Weingarten formula $\bar{\nabla}_X\xi=-S_\xi X+\nabla^\perp_X\xi$.

To a large extent, the geometry of $M$ is governed by the Gauss, Codazzi and Ricci equations, that can be written as
\begin{align*}
\langle\bar{R}(X,Y)Z,W\rangle={}
&\langle{R}(X,Y)Z,W\rangle-\langle\II(Y,Z),\II(X,W)\rangle+\langle\II(X,Z),\II(Y,W)\rangle,\\
\langle\bar{R}(X,Y)Z,\xi\rangle={}
&\langle\nabla_X S_\xi Y,Z\rangle-\langle \nabla_X Y, S_\xi Z\rangle
-\langle S_{\nabla^\perp_X\xi} Y,Z\rangle\\
&{}-\langle\nabla_Y S_\xi X,Z\rangle+\langle \nabla_Y X, S_\xi Z\rangle
+\langle S_{\nabla^\perp_Y\xi} X,Z\rangle,\\
\langle\bar{R}(X,Y)\xi,\eta\rangle={}
&\langle R^\perp(X,Y)\xi,\eta\rangle-\langle[S_\xi,S_\eta]X,Y\rangle,
\end{align*}
where $X$, $Y$, $Z$, $W\in\Gamma(TM)$, $\xi$, $\eta\in\Gamma(\nu M)$ and $R^\perp$ denotes the curvature tensor of $\nabla^\perp$.

We say that $M$ has {flat normal bundle} if $R^\perp=0$. This is equivalent to requiring that each point has a neighborhood where there is an orthonormal frame of $\nu M$ consisting of $\nabla^\perp$-parallel normal vector fields.

It is easy to see that the shape operator with respect to a unit normal vector field $\xi$ is self-adjoint, that is, $\langle S_\xi X,Y\rangle=\langle S_\xi Y,X\rangle$ for all $X$, $Y\in TM$. Hence, each $S_\xi$ is diagonalizable with real eigenvalues and orthogonal eigenspaces. These eigenvalues are called the {principal curvatures} in the direction of $\xi$. The {mean curvature} in the direction of $\xi$ is $\frac{1}{k}\tr S_\xi$, while the {mean curvature vector} of $M$ is defined by $H=\frac{1}{k}\sum_{i=1}^k\II(E_i,E_i)=\frac{1}{k}\sum_{i=1}^l(\tr S_{\xi_i})\xi_i$, where $\{E_1,\dots,E_k\}$ and $\{\xi_1,\dots,\xi_l\}$ are orthonormal frames of $TM$ and of $\nu M$, respectively.

A submanifold $M$ is called totally umbilical if $\II(X,Y)=\langle X,Y\rangle H$ for any $X$, $Y\in TM$ and totally geodesic if $\II=0$. Totally geodesic submanifolds of $\bar{M}^2(c)$ can be geodesics, real projective or hyperbolic planes $\R P^2$ or $\R H^2$, and complex projective or hyperbolic lines $\C P^1$ or $\C H^1$, depending on the sign of the holomorphic curvature~$c$.

\subsection{Isoparametric submanifolds}

Let $\bar{M}$ be a Riemannian manifold and $M$ a submanifold of $\bar{M}$.
The submanifold $M$ is said to be \emph{almost isoparametric}~\cite{HLO06} if its normal bundle $\nu M$ is flat and if, locally, the parallel submanifolds  of $M$ have constant mean curvature in radial directions.

The submanifold $M$ is said to admit sections if for any point $p\in M$ there is a totally geodesic submanifold $\Sigma_p$, called the \emph{section} through $p$, such that $T_p \Sigma_p=\nu_p M$. Then, we say $M$ is \emph{isoparametric}
if it is almost isoparametric and admits sections. Throughout this paper whenever we consider an isoparametric submanifold, we understand that it is isoparametric according to this definition.

The submanifold $M$ is said to have \emph{constant principal curvatures} if for any curve $\sigma\colon I\to M$ and any parallel unit normal vector field $\xi\in\Gamma(\sigma^*\nu^1 M)$ along $\sigma$ the eigenvalues of the shape operator $S_{\xi(t)}$ with respect to $\xi(t)$ are constant along $\sigma$. Then, $M$ is called \emph{Terng-isoparametric} (or isoparametric according to Terng~\cite{T85}) if it has constant principal curvatures and flat normal bundle.

\section{Chen's surface}\label{sec:Chen}

In this section we give a Lie theoretic description of the surface that arises in Theorem~B~(ii). This surface was introduced by Chen in~\cite{C98}.

First we recall the characterizing properties of this surface according to~\cite{C98}. A surface $M$ in $\C H^2$ is called slant if its tangent space has constant K\"{a}hler angle (called Wirtinger angle or slant angle in~\cite{C98}), that is, if for each nonzero vector $v\in T_p M$ the angle between $Jv$ and $T_p M$ is independent of $p\in M$ and $v\in T_p M$. Such surface is called proper slant if it is neither complex nor totally real, that is, if the K\"{a}hler angle is neither $0$ nor $\pi/2$. The Chen's surface that appears in Theorem B~(ii) is, according to~\cite[Theorem~A]{C98} and~\cite[Theorem~5.1]{CT00}, the unique (up to isometric congruence) proper slant surface of $\C H^2$ with K\"{a}hler angle $\theta=\arccos(1/3)$, and satisfying $\langle H,H\rangle=2K-c(1+3\cos^2\theta)/2$, where $K$ is the Gaussian curvature of $M$.

Chen's surface turns out to be homogeneous, although not an orbit of a polar action, and the aim of this section is to give a subgroup of the isometry group of $\C H^2$ one of whose orbits is precisely the Chen's surface. For that matter let $G=SU(1,2)$ and $K=S(U(1)U(2))\subset G$, and write $\C H^2=G/K$. Then $K$ is the isotropy group of $G$ at the origin $o=1K$. We denote by $\g{g}$ and~$\g{k}$ the Lie algebras of $G$ and $K$ respectively. We have the Cartan decomposition with respect to $o$, $\g{g}=\g{k}\oplus\g{p}$, where $\g{p}$ is the orthogonal complement of $\g{k}$ in $\g{g}$ with respect to the Killing form of $\g{g}$. Let $\g{a}$ be a maximal abelian subspace of $\g{p}$, which is known to be $1$-dimensional. For a covector $\lambda\in\g{a}^*$ we define $\g{g}_\lambda=\{X\in\g{g}:[H,X]=\lambda(H)X, \forall H\in\g{a}\}$. Then, one can write $\g{g}=\g{g}_{-2\alpha}\oplus\g{g}_{-\alpha}\oplus
\g{g}_0\oplus\g{g}_\alpha\oplus\g{g}_{2\alpha}$, the so-called root space decomposition of $\g{g}$ with respect to~$o$ and~$\g{a}$. It is known that $\g{g}_0=\g{k}_0\oplus\g{a}$, where $\g{k}_0=\g{g}_0\cap\g{k}$, and that $\g{g}_{2\alpha}$ is $1$-dimensional. We determine an ordering in $\g{a}^*$ so that $\alpha$ is positive, and define the nilpotent subalgebra $\g{n}=\g{g}_\alpha\oplus\g{g}_{2\alpha}$; we denote by $N$ the connected subgroup of $G$ whose Lie algebra is $\g{n}$. The subspace $\g{a}\oplus\g{n}$ is then a solvable subalgebra of $\g{g}$ and we denote by $AN$ the connected subgroup of $G$ whose Lie algebra is $\g{a}\oplus\g{n}$. One can show that $AN$ acts simply transitively on $\C H^2$, and that the metric of $\C H^2$ induces a left-invariant metric in $AN$ that we denote by $\langle\,\cdot\,,\,\cdot\,\rangle$. We also denote by $J$ the complex structure in $\g{a}\oplus\g{n}$ induced by the complex structure of $T_o\C H^2$. This turns $\g{a}\oplus\g{n}$ into a complex vector space such that $\g{g}_\alpha$ is $J$-invariant (that is, $\g{g}_\alpha\cong\C$), and $J\g{a}=\g{g}_{2\alpha}$. Moreover, the decomposition $\g{a}\oplus\g{g}_\alpha\oplus\g{g}_{2\alpha}$ is orthogonal. We choose a unit vector $B\in\g{a}$ and define $Z=JB\in\g{g}_{2\alpha}$. The Levi-Civita connection of $AN$ in terms of left-invariant vector fields is determined by
\begin{equation}\label{eq:Levi-Civita}
\begin{aligned}
\frac{1}{\sqrt{-c}}\bar{\nabla}_{aB+U+xZ}\bigl(bB+V+yZ\bigr)={}
&\Bigl(xy+\frac{1}{2}\langle U,V\rangle\Bigr)B
-\frac{1}{2}\Bigl(bU+yJU+xJV\Bigr)\!\!\\
&{}+\Bigl(-bx+\frac{1}{2}\langle JU,V\rangle\Bigr)Z,
\end{aligned}
\end{equation}
where $a$, $b$, $x$, $y\in\R$, and $U$, $V\in\g{g}_\alpha$. See for example~\cite{BD09}.

Now assume that $V\in\g{g}_\alpha$ is a unit vector. We have $\g{g}_\alpha=\R V\oplus\R JV$. We define the following subalgebra of $\g{a}\oplus\g{n}$:
\[
\g{h}=\R U_1\oplus\R U_2,
\quad\text{ with }\quad
U_1=\frac{1}{\sqrt{3}}\Bigl(\sqrt{2}B+JV\Bigr),\quad\text{ and }\quad
U_2=\frac{1}{\sqrt{3}}\Bigl(V+\sqrt{2}Z\Bigr).
\]
Let $H$ be the connected subgroup of $AN$ whose Lie algebra is $\g{h}$, and $M=H\cdot o$ the orbit through the origin. Since $AN$ acts simply transitively on $\C H^2$ we may identify $H$ with $M$ for the calculations that follow.

First notice that $\{U_1,U_2\}$ is an orthonormal basis of the tangent space of $M$, and $\langle JU_1,U_2\rangle=1/3$. By homogeneity we conclude that $M$ is a proper slant surface with K\"{a}hler angle $\theta=\arccos(1/3)$. Using~\eqref{eq:Levi-Civita} we get the mean curvature vector and the Gaussian curvature
\[
H=\frac{\sqrt{-c}}{3}\Bigl(B-\sqrt{2}JV\Bigr),\quad\text{ and }\quad K=\frac{c}{6}.
\]
It readily follows from this equation that $\langle H,H\rangle=2K-c(1+3\cos^2\theta)/2$ and hence,~\cite[Theorem~A]{C98} and~\cite[Theorem~5.1]{CT00} imply that $M$ is isometrically congruent to the Chen's surface.

\section{Proof of Theorem A}\label{sec:theoremA}

Let $M$ be an isoparametric submanifold of $\bar{M}^2(c)$. By definition, $M$ has a section at every point, that is, for each $p\in M$ there exists a totally geodesic submanifold $\Sigma_p$ through $p$ such that $T_p\Sigma_p=\nu_p M$. Totally geodesic submanifolds of complex space forms are known to be either complex or totally real.

First we assume that the section is complex. Then, $M$ is an almost complex submanifold of a K\"ahler manifold, and hence, $M$ is K\"ahler. Since the normal bundle of $M$ is flat, \cite[Theorem~19]{ADS04} implies that $M$ is either a point or an open part of $\bar{M}^2(c)$.

Hence, we may assume from now on that sections are totally real. In this case, sections are either geodesics or totally geodesic real projective planes $\R P^2$ in $\C P^2$ or real hyperbolic planes $\R H^2$ in $\C H^2$. If the section is a geodesic, $M$ is an isoparametric hypersurface. The classification of isoparametric hypersurfaces in $\C P^2$ follows from~\cite{Do16}, and all examples are open parts of orbits of cohomogeneity one actions. Indeed, we get from~\cite{Do16} the full classification of isoparametric submanifolds of $\C P^2$, but the arguments that follow for higher codimension are also valid for this case. Isoparametric hypersurfaces in $\C H^n$ have been classified in~\cite[Corollary~1.2]{DDS} and it follows from here that $M$ is an open part of a principal orbit of a cohomogeneity one action on~$\C H^2$.

Therefore, we can assume that $M$ has codimension $2$. Since in this case sections are totally real, it follows that $TM$ and $\nu M$ are both totally real. Indeed, $M$ is Lagrangian as $J T_p M=\nu_p M$ for each $p\in M$.

If $M$ is totally umbilical, then it follows from~\cite{CO74} that $M$ is an open part of a totally geodesic real projective plane $\R P^2$ in $\C P^2$ or a totally geodesic real hyperbolic plane $\R H^2$ in $\C H^2$. However, these are not isoparametric because their normal bundles are not flat.

We denote by $\nu^1 M$ the unit normal bundle of $M$. By assumption $\nu M$ is flat. For a given parallel unit normal vector field $\xi\in\Gamma(\nu^1 M)$ and $r>0$ we define $\Phi^{r,\xi}\colon M\to \bar{M}^2(c)$, $p\mapsto \exp_p(r\xi)$. Let $\gamma_{\xi_p}$ be the geodesic of $\bar{M}^2(c)$ with initial conditions $\gamma_{\xi_p}(0)=p$, $\gamma_{\xi_p}'(0)=\xi_p$. We also define the vector field $\eta^r$ along $\Phi^{r,\xi}$ by $\eta^r(p)=\gamma_{\xi_p}'(r)$.
Parallel submanifolds to $M$ are of the form $M^{r,\xi}=\Phi^{r,\xi}(M)$. We calculate their mean curvature at $\Phi^{r,\xi}(p)$ in the direction of~$\eta^r(p)$.

We denote by $\lambda_1,\lambda_2\colon \nu^1 M\to \R$ the principal curvature functions, which are given by the fact that $\lambda_1(\xi)$ and $\lambda_2(\xi)$ are the eigenvalues of the shape operator $S_\xi$. We have already seen that $M$ cannot be totally umbilical, so we may assume that there exists $\xi\in\nu^1 M$ such that $\lambda_1(\xi)\neq\lambda_2(\xi)$. By continuity, the principal curvature functions are thus different on an open neighborhood of $\xi$ in $\nu^1 M$. In the sequel we assume that calculations take place in such a neighborhood. We also denote by $U_1(\xi)$ and $U_2(\xi)$ a (local) orthonormal frame of $TM$ consisting of principal curvature vectors associated with $\lambda_1(\xi)$ and $\lambda_2(\xi)$.

Let $p\in M$. Using standard Jacobi field theory, we get that $\Phi^{r,\xi}_{*p}(v)=X_{v}(r)$ for each $v\in T_p M$, where $X_{v}$ denotes the Jacobi vector field along $\gamma_\xi$ with initial conditions $X_{v}(0)=v$ and $X_{v}'(0)=-S_\xi(v)$. Here $(\cdot)'$ stands for covariant derivative along $\gamma_\xi$. Recall that the Jacobi equation on $\bar{M}^2(c)$ along $\gamma_\xi$ can be written as $4X''+cX+3c\langle X,J\gamma_\xi'\rangle J\gamma_\xi'=0$. Moreover, it is known that the points where a Jacobi field vanishes correspond to the singularities of the Riemannian exponential map. Since the Riemannian exponential map is a local diffeomorphism, it is then clear that $\Phi^{r,\xi}$ is a local diffeomorphism for sufficiently small values of $r$. Thus, we will take, if necessary, a sufficiently small neighborhood of $p$ and sufficiently small values of $r$ so that $\Phi^{r,\xi}$ is a diffeomorphism.

In order to simplify notation we define $u_i=U_i(\xi_p)$, $i=1,2$, and set $v=u_i$ in the previous calculations. Then $X_{u_i}(t)=f_{\lambda_i}(t)\Pa^\xi_{u_i}(t)+\langle u_i,J\xi\rangle g_{\lambda_i}(t)J\gamma_\xi'(t)$, where $\Pa^\xi_v(t)$ denotes parallel transport of $v\in T_p M$ along the geodesic $\gamma_\xi$. The functions $f_\lambda$ and $g_\lambda$ are defined by
\begin{align*}
f_\lambda(t)
&{}=\cosh\Bigl(\frac{t\sqrt{-c}}{2}\Bigr)
-\frac{2\lambda}{\sqrt{-c}}\sinh\Bigl(\frac{t\sqrt{-c}}{2}\Bigr),\\
g_\lambda(t)
&{}=\Bigl(\cosh\Bigl(\frac{t\sqrt{-c}}{2}\Bigr)-1\Bigr)
\Bigl(1+2\cosh\Bigl(\frac{t\sqrt{-c}}{2}\Bigr)
-\frac{2\lambda}{\sqrt{-c}}\sinh\Bigl(\frac{t\sqrt{-c}}{2}\Bigr)\Bigr).
\end{align*}
(For $c>0$ one would have to replace hyperbolic trigonometric functions by standard trigonometric functions.) In other words, $X_{u_i}$ is the parallel transport along $\gamma_\xi$ of the tangent vector $f_{\lambda_i}u_i+\langle u_i,J\xi\rangle g_{\lambda_i}J\xi$. At this point we recall that $M$ has totally real tangent and normal bundles. Thus, $J\xi$ is tangent to $M$ and can be written as $J\xi=\langle U_1(\xi),J\xi\rangle U_1(\xi)+\langle U_2(\xi),J\xi\rangle U_2(\xi)$. Moreover, since $T_{\Phi^{r,\xi}(p)} M^{r,\xi}=\Phi^{r,\xi}_{*p}(T_p M)$ and $\Phi^{r,\xi}$ is a diffeomorphism, it is then clear that $T_{\Phi^{r,\xi}(p)} M^{r,\xi}=\Pa^\xi_{T_p M}(r)$, that is, the tangent space of $M^{r,\xi}$ at the point $\Phi^{r,\xi}(r)$ is obtained by parallel translation of $T_p M$ along the geodesic $\gamma_\xi$ from $p=\gamma_\xi(0)$ to $\Phi^{r,\xi}(r)=\gamma_\xi(r)$.

The previous considerations allow us to define the endomorphism-valued map of the tangent space $D_\xi(t)\colon T_{\Phi^{t,\xi}(p)} M^{t,\xi}\to T_{\Phi^{t,\xi}(p)} M^{t,\xi}$ by $D_\xi(t)(\Pa^\xi_v(t))=X_v(t)$, for each $v\in T_p M$. Since we are assuming that $r$ is sufficiently small,
$D_\xi(r)$ is actually an isomorphism of the tangent space. We denote now by $S^{r,\xi}_{\eta^r}$ the shape operator of $M^{r,\xi}$ with respect to the radial vector $\eta^r$. It follows from Jacobi field theory that the shape operator of $M^{r,\xi}$ is given by $S^{r,\xi}_{\eta^r}(\Phi^{r,\xi}_{*p}(v))=-X_v'(r)^\top$, where $(\cdot)^\top$ denotes the orthogonal projection onto the tangent space $T_{\Phi^{r,\xi}(p)} M^{r,\xi}$. By the previous calculations, $X_{u_i}'(t)=f_{\lambda_i}'(t)\Pa^\xi_{u_i}(t)+\langle u_i,J\xi\rangle g_{\lambda_i}'(t)J\gamma_\xi'(t)\in T_{\Phi^{r,\xi}(p)} M^{r,\xi}$. This implies that $S^{r,\xi}_{\eta^r}=-D_\xi'(r)D_\xi(r)^{-1}$. Finally, the mean curvature in radial directions is the function $h^{r,\xi}\colon M^{r,\xi}\to\R$ determined by
\[
h^{r,\xi}(\Phi^{r,\xi}(p))=\frac{1}{2}\tr S^{r,\xi}_{\eta^r(p)}
=-\frac{1}{2}\tr D_\xi'(r)D_\xi(r)^{-1}
=-\frac{\frac{d}{dr}\det D_\xi(r)}{2\det D_\xi(r)}.
\]
It is easy to check that $\det D_\xi=f_{\lambda_1}f_{\lambda_2}+\langle U_1(\xi),J\xi\rangle^2 f_{\lambda_2}g_{\lambda_1}+\langle U_2(\xi),J\xi\rangle^2 f_{\lambda_1}g_{\lambda_2}$. The function $h^{r,\xi}\circ\Phi^{r,\xi}$ can be calculated explicitly, but for our purpose it suffices to calculate its Taylor power series expansion. After some relatively long but elementary calculations, and using $\langle U_1(\xi),J\xi\rangle^2
+\langle U_2(\xi),J\xi\rangle^2=\langle J\xi,J\xi\rangle=1$, we get
\begin{align*}
h^{r,\xi}(\Phi^{r,\xi}(p))={}&
\frac{1}{2}\bigl(\lambda_1(\xi_p)+\lambda_2(\xi_p)\bigr)
+\frac{r}{2}\Bigl(\frac{5c}{4}+\lambda_1(\xi_p)^2+\lambda_2(\xi_p)^2\Bigr)\\
&{}+\frac{r^2}{8}\Bigl(c\bigl(\lambda_1(\xi_p)+\lambda_2(\xi_p)\bigr)
+4\bigl(\lambda_1(\xi_p)^3+\lambda_2(\xi_p)^3\bigr)\\
&\phantom{{}+\frac{r^2}{4}\Bigl(}
+3c\bigl(\lambda_1(\xi_p)\langle U_1(\xi_p),J\xi_p\rangle^2
+\lambda_2(\xi_p)\langle U_2(\xi_p),J\xi_p\rangle^2\bigr)\Bigr)+O(r^3).
\end{align*}
Since $M$ is isoparametric, the function $h^{r,\xi}$ is constant by assumption. Since $\Phi^{r,\xi}$ is a diffeomorphism, this is equivalent to requiring that $(h^{r,\xi}\circ\Phi^{r,\xi})(p)$ does not depend on the point~$p$. More precisely, by hypothesis the expression $(h^{r,\xi}\circ\Phi^{r,\xi})(p)$ depends both on $r$ and on the choice of parallel unit normal vector field $\xi\in\Gamma(\nu^1 M)$, but not on the base point $p$ of the vector $\xi_p$. Therefore, using the above power series expansion we obtain that the functions $p\mapsto \lambda_i(\xi)(p)=\lambda_i(\xi_p)$, and $p\mapsto \langle U_i(\xi),J\xi\rangle(p)=\langle U_i(\xi_p),J\xi_p\rangle$, $i=1,2$, are constant for a fixed parallel vector field $\xi\in\Gamma(\nu^1 M)$. By linearity this argument readily implies:

\begin{proposition}\label{prop:isoparametric-Terng}
An isoparametric submanifold of $\bar{M}^2(c)$ is Terng-isoparametric.
\end{proposition}

In order to conclude the proof of Theorem A we simply have to verify the following assertion:

\begin{proposition}\label{prop:Lagrangian-Terng}
Let $M$ be a Lagrangian, Terng-isoparametric submanifold of $\bar{M}^2(c)$. Then, $M$ is an open part of a principal orbit of a cohomogeneity two polar action on $\bar{M}^2(c)$.
\end{proposition}

\begin{proof}
Since $M$ is Lagrangian, $J\nu_p M=T_p M$. Let $\xi\in\Gamma(\nu M)$ be a parallel normal vector field and $X\in\Gamma(TM)$. We denote by $(\cdot)^\perp$ the orthogonal projection onto~$\nu M$. As $J\xi$ is tangent and since $\bar{M}^2(c)$ is K\"{a}hler, the definition of the second fundamental form yields
\begin{align*}
0 &{}=\nabla_X^\perp \xi=-\nabla^\perp_X J^2\xi
=-\bigl(\bar{\nabla}_X J^2\xi\bigr)^\perp\\
&{}=-\bigl(J\bar{\nabla}_X J\xi\bigr)^\perp
=-\bigl(J(\nabla_X J\xi+\II(X,J\xi))\bigr)^\perp
=-J\nabla_X J\xi.
\end{align*}
Therefore, $\nabla J\xi=0$ and it follows that $M$ is flat.

Since $M$ has constant principal curvatures and flat normal bundle, it is clear that $M$ has parallel mean curvature. Thus, $M$ is a Lagrangian, flat surface of $\bar{M}^2(c)$ with parallel mean curvature and it was shown in~\cite[Theorem~2.1]{DDV} 
that $M$ is then an open part of a principal orbit of a cohomogeneity two polar action on $\bar{M}^2(c)$.
\end{proof}

In particular, propositions~\ref{prop:isoparametric-Terng} and~\ref{prop:Lagrangian-Terng}, together with the fact that the principal orbits of a polar action are isoparametric submanifolds implies

\begin{corollary}
A Lagrangian submanifold of $\bar{M}^2(c)$ is isoparametric if and only if it is Terng-isoparametric.
\end{corollary}

\begin{remark}
There is a shorter alternative proof of Theorem A that does not require working with Jacobi fields. Indeed, once the problem was reduced to the case of an isoparametric Lagrangian surface $M$, we could have argued as in the proof of Proposition~\ref{prop:Lagrangian-Terng} to show that $M$ is flat. Since by assumption $M$ is Lagrangian and has parallel mean curvature, by virtue of~\cite{DDV}, $M$ is a piece of a principal orbit of a cohomogeneity two polar action. However, we have preferred to include the longer argument because it shows that, in order to prove that an isoparametric submanifold in $\bar{M}^2(c)$ has constant principal curvatures, it is not necessary to appeal to the strong result in~\cite{DDV}.
\end{remark}

\section{Proof of Theorem B}\label{sec:theoremB}

We now consider a Terng-isoparametric submanifold $M$ of $\bar{M}^2(c)$. In particular, the normal bundle of $M$ is flat, and we have already seen in Section~\ref{sec:theoremA} that, if the normal bundle of $M$ is complex, then $M$ is either a point or an open subset of $\bar{M}^2(c)$. Thus, we may assume from now on that the normal bundle of $M$ is not complex.

If the normal bundle of $M$ is totally real, $M$ is either a hypersurface or a Lagrangian submanifold. In the first case, $M$ is a hypersurface of $\bar{M}^2(c)$ with constant principal curvatures. These were classified in~\cite{W83} for $\C P^2$ and in~\cite{BD07} for $\C H^2$ where it was shown that such hypersurfaces are open parts of homogeneous hypersurfaces. In particular they are open parts of orbits of cohomogeneity one actions, which are polar.

If the normal bundle is totally real and has rank $2$, then $M$ is Lagrangian. Hence, it follows from Proposition~\ref{prop:Lagrangian-Terng} that $M$ is an open part of a principal orbit of a cohomogeneity two polar action on $\bar{M}^2(c)$.

Therefore, we can assume from now on that the normal bundle of $M$ is neither complex nor totally real. If $M$ is $1$-dimensional, then $M$ has to be a geodesic or a circle~\cite[\S8.4]{BCO03}, so we also assume that $M$ is $2$-dimensional.

Hence, we take, at least locally, a parallel orthonormal frame $\{\xi,\eta\}$ of the normal bundle of $M$, and let $\{U_1,U_2\}$ be an orthonormal frame of the tangent bundle of $M$ such that $S_\xi U_i=\lambda_i U_i$, $i=1,2$. Since $\xi$ is parallel, $\lambda_1$ and $\lambda_2$ are constant by assumption. At this point we observe that the mean curvature vector of $M$ is parallel because the normal bundle is flat and the principal curvatures are constant (and hence the trace of each shape operator with respect to a parallel normal vector field is constant). Therefore, we may further assume that $\{\xi,\eta\}$ is chosen so that $\eta$ is perpendicular to the mean curvature vector field.

Using the fact that $TM$ and $\nu M$ are neither complex nor totally real we can write $J\xi=b_1 U_1+b_2 U_2+a\eta$, where $a,b_1,b_2\colon M\to \R$ are smooth functions with $b_1^2+b_2^2+a^2=1$, and $b_1^2+b_2^2\neq 0$, $a\neq 0$. Since $\{U_1,U_2,\xi,\eta\}$ is an orthonormal frame of $T\C H^2$ we can write
\begin{align*}
-\xi={} &J^2\xi=b_1JU_1+b_2JU_2+aJ\eta\\
{}={}& b_1(\langle JU_1,U_2\rangle U_2-b_1\xi+\langle JU_1,\eta\rangle\eta)
+b_2(-\langle JU_1,U_2\rangle U_1-b_2\xi+\langle JU_2,\eta\rangle\eta)\\
& {}+a(-\langle JU_1,\eta\rangle U_1-\langle JU_2,\eta\rangle U_2-a\xi)\\
{}={} & (-b_2\langle JU_1,U_2\rangle-a\langle JU_1,\eta\rangle)U_1
+(b_1\langle JU_1,U_2\rangle-a\langle JU_2,\eta\rangle)U_2\\
&{}+(b_1\langle JU_1,\eta\rangle+b_2\langle JU_2,\eta\rangle)\eta-\xi.
\end{align*}
Thus, we have
\[
-b_2\langle JU_1,U_2\rangle-a\langle JU_1,\eta\rangle=
b_1\langle JU_1,U_2\rangle-a\langle JU_2,\eta\rangle=
b_1\langle JU_1,\eta\rangle+b_2\langle JU_2,\eta\rangle=0.
\]
Using these equalities and $b_1^2+b_2^2+a^2=1$, it is easy to show that we can write (up to a choice of orientation)
\begin{align*}
J\xi & {}=b_1 U_1+b_2 U_2+a\eta,
& J\eta & {}=-b_2U_1+b_1U_2-a\xi,\\
JU_1 & {}=-a U_2-b_1 \xi+b_2\eta,
& JU_2 & {}=a U_1-b_2 \xi-b_1\eta.
\end{align*}

For $i\in\{1,2\}$, using the Codazzi equation, taking into account that $\xi$ is parallel and that $\lambda_1$ and $\lambda_2$ are constant, we get
\begin{align*}
-\frac{3cab_i}{4}&{}=\langle\bar{R}(U_1,U_2)U_i,\xi\rangle
=(\lambda_2-\lambda_i)\langle\nabla_{U_1}U_2,U_i\rangle
-(\lambda_1-\lambda_i)\langle\nabla_{U_2}U_1,U_i\rangle.
\end{align*}
Since $a\neq 0$ and $b_1^2+b_2^2\neq 0$, we readily get $\lambda_1\neq\lambda_2$.
Since $\{U_1,U_2\}$ is an orthonormal frame of the tangent bundle we obtain
\begin{align}\label{eq:nablaUiUj}
\nabla_{U_i}U_i
&{}=-\frac{3cab_i}{4(\lambda_1-\lambda_2)}U_j,
&\nabla_{U_i}U_j
&{}=\frac{3cab_i}{4(\lambda_1-\lambda_2)}U_i,
& i,j\in\{1,2\},i\neq j.
\end{align}

Now, since $\nu M$ is flat, the Ricci equation implies
\begin{align*}
\frac{c}{4}(-b_1^2-b_2^2+2a^2)
&{}=\langle\bar{R}(U_1,U_2)\xi,\eta\rangle
=\langle S_\xi U_1,S_\eta U_2\rangle
-\langle S_\eta U_1,S_\xi U_2\rangle\\
&{}=(\lambda_1-\lambda_2)\langle S_\eta U_1,U_2\rangle.
\end{align*}
Recall that, since $\eta$ is perpendicular to the mean curvature vector, we have $\tr S_\eta=0$,
and thus, with respect to the orthonormal basis $\{U_1,U_2\}$ the shape operator $S_\eta$ can be written as
\begin{equation}\label{eq:Seta}
S_\eta=
\begin{pmatrix}
\mu & -\frac{c(1-3a^2)}{4(\lambda_1-\lambda_2)}\\
-\frac{c(1-3a^2)}{4(\lambda_1-\lambda_2)} & -\mu
\end{pmatrix}.
\end{equation}
for some function $\mu\colon M\to \R$.

By assumption, the eigenvalues of $S_\eta$ are constant, or equivalently, the functions
\begin{equation}\label{eq:trSeta}
\tr S_\eta=0\quad\text{ and }\quad
\tr S_\eta^2
=2\mu^2+\frac{c^2(1-3a^2)^2}{8(\lambda_1-\lambda_2)^2}
\end{equation}
are constant.

Now we calculate the derivatives of $b_1$, $b_2$ and $a$. We take $i,j\in\{1,2\}$, $i\neq j$. Using~\eqref{eq:nablaUiUj} and~\eqref{eq:Seta} we obtain
\begin{equation}\label{eq:derivatives}
\begin{aligned}
U_ib_i
&{}=U_i\langle U_i,J\xi\rangle
=\langle\bar{\nabla}_{U_i}U_i,b_i U_i+b_j U_j+a\eta\rangle
+\langle U_i,\bar{\nabla}_{U_i}J\xi\rangle\\
&{}=b_j\langle\nabla_{U_i}U_i,U_j\rangle
+a\langle U_i,S_\eta U_i\rangle
-\lambda_i\langle U_i,JU_i\rangle
=-\frac{3cab_1b_2}{4(\lambda_1-\lambda_2)}-a(-1)^{i}\mu,\\[1ex]
U_ib_j
&{}=U_i\langle U_j,J\xi\rangle
=\langle\bar{\nabla}_{U_i}U_j,b_i U_i+b_j U_j+a\eta\rangle
+\langle U_j,\bar{\nabla}_{U_i}J\xi\rangle\\
&{}=b_i\langle\nabla_{U_i}U_j,U_i\rangle
+a\langle U_j,S_\eta U_i\rangle
-\lambda_i\langle U_j,JU_i\rangle\\
&{}=\frac{3cab_i^2}{4(\lambda_1-\lambda_2)}
-\frac{ca(1-3a^2)}{4(\lambda_1-\lambda_2)}-a(-1)^{i}\lambda_i,\\[1ex]
U_ia
&{}=U_i\langle J\xi,\eta\rangle
=\langle\bar{\nabla}_{U_i}J\xi,\eta\rangle
+\langle b_iU_i+b_jU_j+a\eta,\bar{\nabla}_{U_i}\eta\rangle\\
&{}=-\lambda_i\langle JU_i,\eta\rangle
-b_i\langle U_i,S_\eta U_i\rangle
-b_j\langle U_j,S_\eta U_i\rangle\\
&{}=b_j(-1)^i\lambda_i+b_i(-1)^i\mu+\frac{cb_j(1-3a^2)}{4(\lambda_1-\lambda_2)}.
\end{aligned}
\end{equation}

In order to get a relation for the derivatives of $\mu$, we use the Codazzi equation together with~\eqref{eq:nablaUiUj},~\eqref{eq:Seta} and~\eqref{eq:derivatives} to get, after some calculations
\begin{align*}
-\frac{3c(-1)^i a b_j}{4}
&{}=\langle\bar{R}(U_1,U_2)U_i,\eta\rangle\\
&{}=\langle\nabla_{U_1}S_\eta U_2,U_i\rangle
-\langle\nabla_{U_1}U_2,S_\eta U_i\rangle
-\langle\nabla_{U_2} S_\eta U_1,U_i\rangle
+\langle\nabla_{U_2}U_1,S_\eta U_i\rangle\\
&{}=-U_j\mu-\frac{3ca(b_j\lambda_i+2b_i\mu)}{2(\lambda_1-\lambda_2)}.
\end{align*}
Thus, we obtain
\begin{equation}\label{eq:Uimu}
U_i\mu=\frac{3ca}{4(\lambda_1-\lambda_2)}(b_i\lambda_i-3b_i\lambda_j-4b_j\mu),\quad i,j\in\{1,2\},i\neq j.
\end{equation}

The aim of the argument that follows is to show that the functions $b_1$, $b_2$, $a$ and $\mu$ are constant. We first have

\begin{lemma}\label{lemma:constant}
If the function $a\colon M\to\R$ is constant, then $b_1$, $b_2$ and $\mu$ are also constant. \end{lemma}

\begin{proof}
If $a$ is constant, it readily follows from~\eqref{eq:trSeta} that $\mu$ is constant. Hence, from~\eqref{eq:Uimu} we get $(\lambda_1-3\lambda_2)b_1-4\mu b_2=-4\mu b_1+(\lambda_2-3\lambda_1)b_2=0$. This is a homogeneous linear system in the variables $b_1$ and $b_2$, whose coefficients are constant. It cannot have a unique solution because $b_1=b_2=0$ is not possible, and thus the rank of the matrix of the system cannot be~$2$. The rank cannot be $0$ because that would imply $\lambda_1=\lambda_2=0$. Thus, it has rank one and we can write $b_2=\nu b_1$ for some constant $\nu\in \R$. Then $1-a^2=b_1^2+b_2^2=(1+\nu^2)b_1^2$ implies that $b_1$ is constant, and hence also $b_2$.
\end{proof}

In view of Lemma~\ref{lemma:constant}, the calculations that follow aim at proving that $a$ is constant. Recall from~\eqref{eq:trSeta} that $\tr S_\eta^2$ is constant. Hence there is $k\in\R$ such that
\begin{equation}\label{eq:mu}
\mu^2=k-\frac{c^2(1-3a^2)}{16(\lambda_1-\lambda_2)^2}.
\end{equation}
Taking derivatives in~\eqref{eq:mu} with respect to $U_i$, using~\eqref{eq:derivatives} and~\eqref{eq:Uimu} and substituting $\mu^2$ by~\eqref{eq:mu} we get, after some calculations
\begin{equation}\label{eq:get-mu}
\begin{aligned}
0={}
&{} b_j\bigl((-1)^jc^2(1-3a^2)^2+4c(1-3a^2)\lambda_i(\lambda_1-\lambda_2)
+32(-1)^ik(\lambda_1-\lambda_2)^2\bigr)\\
&{}+4b_i(\lambda_1-\lambda_2)
\bigl(c(1-3a^2)-2(-1)^i(\lambda_1-\lambda_2)(\lambda_i-3\lambda_j)\bigr)\mu.
\end{aligned}
\end{equation}

If
$c(1-3a^2)+2(\lambda_1-\lambda_2)(\lambda_1-3\lambda_2)$ or $c(1-3a^2)-2(\lambda_1-\lambda_2)(\lambda_2-3\lambda_1)$
is zero in an open set, then the function $a$ is constant and it follows from Lemma~\ref{lemma:constant} that $b_1$, $b_2$ and~$\mu$ are also constant. As a consequence, we may assume that there is a point in $M$ where these two functions do not vanish, and thus, they do not vanish in an open set. Moreover, if $b_i=0$ in an open set, then it follows from the first equation in~\eqref{eq:derivatives} that $\mu=0$, so by~\eqref{eq:mu}, $a$ is constant, and thus also $b_j$. Hence, we also assume that $b_i$, $i=1,2$, is not zero on an open set. Thus, from~\eqref{eq:get-mu} we get two possible expressions for $\mu$, and combining this with~\eqref{eq:mu} yields
\[
\scriptsize
\begin{aligned}
0={}&
\Bigl(k-\frac{c^2(1-3a^2)^2}{16(\lambda_1-\lambda_2)^2}\Bigr)-
\Bigl(-b_2\frac{c^2(1-3a^2)^2+4c(1-3a^2)\lambda_1(\lambda_1-\lambda_2)
-32k(\lambda_1-\lambda_2)^2}{4b_1(\lambda_1-\lambda_2)
(c(1-3a^2)+2\lambda_1^2-8\lambda_1\lambda_2+6\lambda_2^2)}\Bigr)\cdot\\
&{}\phantom{\Bigl(k-\frac{c^2(1-3a^2)}{16(\lambda_1-\lambda_2)^2}\Bigr)-{}}
\cdot\Bigl(b_1\frac{c^2(1-3a^2)^2-4c(1-3a^2)\lambda_2(\lambda_1-\lambda_2)
-32k(\lambda_1-\lambda_2)^2}{4b_2(\lambda_1-\lambda_2)
(c(1-3a^2)+6\lambda_1^2-8\lambda_1\lambda_2+2\lambda_2^2)}\Bigr)\\
{}={}
&\frac{-c^3(1-3a^2)^3-3c^2(1-3a^2)^2(4k+(\lambda_1-\lambda_2)^2)
+16k(\lambda_1-\lambda_2)^2(16k+3\lambda_1^2-10\lambda_1\lambda_2+3\lambda_2^2)}{4
(c(1-3a^2)+2\lambda_1^2-8\lambda_1\lambda_2+6\lambda_2^2)
(c(1-3a^2)+6\lambda_1^2-8\lambda_1\lambda_2+2\lambda_2^2)}.
\end{aligned}
\]
This equation implies that $1-3a^2$ is constant, and hence, by Lemma~\ref{lemma:constant} we get that $b_1$, $b_2$ and $\mu$ are also constant.

Using~\eqref{eq:derivatives} we get
\[
0=U_1b_1+U_2b_2=-\frac{3cab_1b_2}{2(\lambda_1-\lambda_2)},
\]
and since $a\neq 0$ we get $b_1=0$ or $b_2=0$. We may assume $b_1\neq 0$, $b_2= 0$, $a^2=1-b_1^2$. Then, by~\eqref{eq:derivatives} we obtain $0=U_2b_2=-a\mu$, so $\mu=0$. Next, equation~\eqref{eq:Uimu} implies that $0=U_1\mu=3cab_1(\lambda_1-3\lambda_2)/(4(\lambda_1-\lambda_2))$, and thus, $\lambda_1=3\lambda_2\neq 0$. Finally, using~\eqref{eq:derivatives} once more,
\[
0=U_1b_2=\frac{3cab_1^2-ca(1-3a^2)}{4(\lambda_1-\lambda_2)}+a\lambda_1
=\frac{a(c+12\lambda_2^2)}{4\lambda_2}.
\]
Hence, if $c>0$ we get a contradiction, which yields

\begin{proposition}
A Terng-isoparametric surface of $\C P^2$ is isoparametric.
\end{proposition}

Otherwise, if $c<0$ we have $\lambda_2=\pm\sqrt{-3c}/6$. By changing the orientation if necessary, we may assume $\lambda_2>0$. Finally,~\eqref{eq:derivatives} yields $0=U_2a=cb_1(9b_1^2-8)/(4\sqrt{-3c})$. Altogether we have obtained
\begin{align*}
S_\xi &{}=\begin{pmatrix}
\frac{\sqrt{-3c}}{2} & 0\\
0   &   \frac{\sqrt{-3c}}{6}
\end{pmatrix},&
S_\eta &{} =\begin{pmatrix}
 0 & \frac{\sqrt{-3c}}{6}\\
 \frac{\sqrt{-3c}}{6} & 0
\end{pmatrix},&
a&{}=\frac{1}{3},
&b_1 &{}=\frac{2\sqrt{2}}{3},
&b_2 &{}=0.
\end{align*}

Finally, it follows from~\cite[Theorem~5.1(vi)]{CT00} that $M$ is an open part of a Chen's surface, as we wanted to show.


\end{document}